\documentclass[11pt]{amsart}
\usepackage{amssymb,amsthm,amsmath,amstext}
\usepackage{graphicx}
\usepackage{mathrsfs}
\usepackage{bm}
\usepackage{multirow}
\usepackage[all]{xy}
\usepackage{tikz-cd}
\usepackage{xfrac}
\usepackage{anyfontsize}

\usepackage[left=1.25in,top=1.35in,right=1.25in,bottom=1.35in]{geometry}

\usepackage{color}
\usepackage{fancyvrb}   
 
\definecolor{airforceblue}{rgb}{0.36, 0.54, 0.66}
\definecolor{darkorchid}{rgb}{0.6, 0.2, 0.8}
\definecolor{darkorange}{rgb}{1.0, 0.55, 0.0}
\definecolor{darkspringgreen}{rgb}{0.09, 0.45, 0.27}

\theoremstyle{plain}
\newtheorem{thm}{Theorem}\newtheorem*{thm*}{Theorem}
\newtheorem*{conj*}{Conjecture}

\newtheorem{pr}[thm]{Proposition}
\newtheorem{lem}[thm]{Lemma}

\theoremstyle{remark}

\theoremstyle{definition}

\theoremstyle{remark}

\newcommand{\ZZ}{\mathbb{Z}}
\newcommand{\QQ}{\mathbb{Q}}

\newcommand{\CC}{\mathbb{C}}

\newcommand{\PP}{\mathbb{P}}

\newcommand{\sheaf}[1]{\mathcal{#1}}

\newcommand{\ko}{\sheaf{O}}

\usepackage[backref=page]{hyperref}
\hypersetup{pdftitle={}}%
\hypersetup{pdfauthor={}}%
\hypersetup{colorlinks=true,linkcolor=blue,anchorcolor=blue,citecolor=blue}

\begin{document}

\title[A note on special cubics of discriminant 14 and non-minimal K3 surfaces]{A note on special cubic fourfolds of discriminant 14 and non-minimal K3 surfaces of degree 10}

\author{Jordi Hern\'andez}
\address{Institut de Math\'ematiques de Toulouse ; UMR 5219 \\ 
UPS, F-31062 Toulouse Cedex 9, France}
\email{jordi\_emanuel.hernandez\_gomez@math.univ-toulouse.fr}

\begin{abstract}
We prove that a general cubic in the Hassett divisor $\mathcal{C}_{14}$ of special cubic fourfolds of discriminant $14$ contains a non-minimal K3 surface of degree $10$ containing two skew $(-1)$-lines and contained in a smooth quadric hypersurface $Q^4\subseteq \PP^5$, but not contained in any other (possibly of lower rank) quadric hypersurface.
\end{abstract}
\maketitle

\section{Introduction}
A special cubic fourfold is a smooth cubic hypersurface in $ \PP ^5$ containing a surface not homologous to a complete intersection. These form a countable family of divisors $\mathcal{C}_{d}$, known as Hassett divisors, in the moduli space of smooth cubic fourfolds indexed by an integer $d$ called discriminant, depending on the degree of the surface and its self-intersection number inside the cubic. 

Although it is expected that a very general cubic fourfold is non-rational, no example of a non-rational cubic has been found as of today. In fact, a famous conjecture of A. Kuznetsov asserts that special cubic fourfolds with a so-called admissible discriminant are exactly the rational cubic fourfolds. The condition on admissibility is purely arithmetical and can be shown to be equivalent to the existence of a polarized K3 surface that can be associated to the cubic fourfold in a Hodge-theoretical way. For example, cubics parametrized by the Hassett divisors $\mathcal{C}_{14}$, $\mathcal{C}_{26}$, $\mathcal{C}_{38}$, and $\mathcal{C}_{42}$ are known to be rational (see \cite{BoRuSt}, \cite{RuSt.fivesecants}, and \cite{RuSt.trisecant}). The first admissible discriminants $d$ are precisely $14$, $26$, $38$, and $42$.

Among all special cubic fourfolds, those of discriminant $14$ are one of the most studied due to their prominent role in the general theory of cubic fourfolds. It is well known that the closure of the locus of Pfaffian cubics, i.e., smooth cubics defined by the Pfaffian of a $6\times 6$ skew-symmetric matrix of linear forms in $\PP ^5$, is the Hassett divisor $\mathcal{C}_{14}$. In fact, a smooth cubic fourfold is Pfaffian if and only if it contains a quintic Del Pezzo surface (see \cite{Beauville.determinantal}). Moreover, $\mathcal{C}_{14}$ can also be characterized as the closure of the locus of smooth cubic fourfolds containing a rational normal quartic scroll. The rich geometry of cubics in $\mathcal{C}_{14}$, e.g., their rationality, is classically tied to the existence of the quartic scrolls and quintic Del Pezzo surfaces. Therefore, it is important to investigate the existence of other surfaces characterizing a general cubic of discriminant $14$.

The aim of this paper is to study a class of non-minimal K3 surfaces of degree $10$ in $\PP^5$ containing two skew $(-1)$-lines and contained in a smooth quadric hypersurface $Q^4\subseteq \PP^5$, but not contained in any other (possibly of lower rank) quadric hypersurface, and their relation to special cubic fourfolds of discriminant $14$. According to our terminology, based on that used by M. Gross in his classification of smooth surfaces of degree $10$ in the four-dimensional smooth quadric in \cite{gross.degree10}, we say that these surfaces are of type $^{II} Z_{E}^{10}$. The main result of this paper is the following theorem.

\begin{thm}\label{characterization of C14}
    A general cubic in $\mathcal{C}_{14}$ contains a surface of type $^{II} Z_{E}^{10}$.
\end{thm}

This is the first example of a family of non-rational surfaces characterizing the Hassett divisor $\mathcal{C}_{14}$. We also investigate the relation of surfaces of type $^{II} Z_{E}^{10}$ and quintic Del Pezzo surfaces. Since rational normal quartic scrolls and quintic Del Pezzo surfaces in a cubic are linked via a (cubic) Segre threefold, we can ask if an analogous situation occurs for surfaces of type $^{II} Z_{E}^{10}$ and quintic Del Pezzo surfaces. We provide a negative answer to this question in the following theorem.

\begin{thm}\label{surfaces of type II are not linked to quintic DP}
    Let $X$ be a smooth cubic fourfold containing a surface $S$ of type $^{II} Z_{E}^{10}$ and a quintic Del Pezzo surface $D$. Then $S=5h^2-D$ in $H^4(X,\ZZ)$, but no smooth quintic threefold $Y\subseteq \PP^5$ satisfies $Y\cap X = S\cup D$.
\end{thm}

The paper is organized as follows. In section \ref{section 1} we show that there exists a smooth open subset of a Hilbert scheme that parametrizes surfaces of type $^{II} Z_{E}^{10}$ in $\PP^5$ and we compute its dimension. In section \ref{section 2} we prove Theorem \ref{characterization of C14} by following the approach of H. Nuer in \cite{nuer17}. The idea of the proof is to consider the Hilbert scheme of flags $[S\subseteq X]$ where $S$ is a surface of type $^{II} Z_{E}^{10}$ and $X$ is a smooth cubic fourfold. Standard deformation theory for Hilbert schemes and semicontinuity arguments imply that if we can exhibit an explicit flag $[\mathtt{S}\subseteq \mathtt{X}]$ such that $h^0(N_{\mathtt{S}/\mathtt{X}})= 15$, then there exists an open neighborhood of $[\mathtt{S}\subseteq \mathtt{X}]$ in the Hilbert scheme of flags that is projected dominantly into the locus of special cubic fourfolds of discriminant $14$. We use \texttt{Macaulay2} in order to exhibit such flag. Finally, in section \ref{section 3} we prove Theorem \ref{surfaces of type II are not linked to quintic DP}. 

\section{Surfaces of type \texorpdfstring{$^{II} Z_{E}^{10}$}{TEXT}}\label{section 1}
In the notation used by M. Gross in \cite{gross.degree10}, surfaces of type $Z_E ^{10}$ are non-degenerate surfaces of degree $10$ in a fixed smooth quadric fourfold $Q^4$ containing two skew $(-1)$-lines whose contraction defines a K3 surface of genus $7$. There is, however, a remarkable difference between the class of surfaces of type $Z_E ^{10}$ that are contained in a quadratic complex of $Q^4$, i.e., a complete intersection of $Q^4$ with another quadric in $\PP^5$, and the class of those surfaces which are not. In \cite[Definition~10]{hernández2024cubocubic} we defined a surface $S\subseteq Q^4$ to be of type $^{II} Z_{E}^{10}$ if it is a surface of type $Z_E ^{10}$ falling in the latter class. More generally, we say that a smooth surface $S\subseteq \PP^5$ is of type $^{II} Z_{E}^{10}$ if there exists a smooth quadric $Q^4$ containing $S$ and this surface is of type $^{II} Z_{E}^{10}$ according to the previous definition. Thus, $h^0(\mathcal{I}_{S/\PP^5}(2))=1$, or equivalently, there is no other quadric hypersurface in $\PP^5$ containing $S$.

Define the polynomial $P(t):=5t^2 - t +2 \in \QQ [t]$. We know that a smooth surface in $Q^4$ with Hilbert polynomial $P(t)$ can only be of type $ Z_B^{10}$ or $ Z_{E}^{10}$ by \cite[Lemma~11]{hernández2024cubocubic}. According to \cite{gross.degree10}, a surface of type $ Z_B^{10}$ is a surface of degree $10$ with a $(-1)$-conic whose contraction is a K3 surface. We extend the definition of surfaces of type $Z_B^{10}$ in $Q^4$ to $\PP^5$ as before. Let ${Hilb}_{Q^4}^{P(t)}$ and ${Hilb}_{\PP^5}^{P(t)}$ be the Hilbert schemes parametrizing closed subschemes of $Q^4$ and $\PP^5$, respectively, with Hilbert polynomial $P(t)$. 

\begin{lem}\label{Hilb is smooth at S type EII and type B} We have the following:
\begin{enumerate}
    \item Let $S$ be a surface of type $ Z_B^{10}$. Then ${Hilb}_{\PP^5}^{P(t)}$ is smooth at $[S]$ and the Zariski tangent space $T_{[S]}{Hilb}_{\PP^5}^{P(t)}$ has dimension $56$.
    \item Let $S$ be a surface of type $^{II} Z_{E}^{10}$. Then ${Hilb}_{\PP^5}^{P(t)}$ is smooth at $[S]$ and the Zariski tangent space $T_{[S]}{Hilb}_{\PP^5}^{P(t)}$ has dimension $58$.
\end{enumerate}
\end{lem}
\begin{proof} 
    Let $S\subseteq Q^4$ be a surface of type $^{II} Z_{E}^{10}$. In \cite[Lemma~12]{hernández2024cubocubic} we proved that $H^1(N_{S/Q^4})=H^2(N_{S/Q^4})=0$, which implies that ${Hilb}_{Q^4}^{P(t)}$ is smooth at $[S]$ and $dim(T_{[S]} {{Hilb}_{Q^4}^{P(t)}} )=dim (H^0(N_{S/Q^4}))=\chi(N_{S/Q^4})=38$. 
    
    Consider the short exact sequence induced from the nested sequence of inclusions of smooth varieties $S\subseteq Q^4\subseteq \PP^5$
    \[0\rightarrow N_{S/Q^4}\rightarrow N_{S/\PP^5}\rightarrow {N_{Q^4/\PP^5}}_{|S}=\ko_S(2)\rightarrow 0,\]
    the tautological short exact sequence 
    \[0\rightarrow \mathcal{I}_{S/ Q^4}\rightarrow \ko_{Q^4}\rightarrow \ko_S\rightarrow 0,\]
    and the resolution of the ideal sheaf of $S$ in $Q^4$ given in \cite[Proposition~4.3]{gross.degree10}
    \[0\rightarrow \mathcal{E}(-3)\oplus \mathcal{E}'(-3)\oplus \ko_{Q^4}(-4)\rightarrow \ko_{Q^4}(-3)^{\oplus 6} \rightarrow \mathcal{I}_{S/Q^4}\rightarrow 0.\]
    Here $\mathcal{E}$ and $\mathcal{E}'$ denote the spinor bundles of $Q^4$.
    The long exact sequences in cohomology induced from the three sequences above allow us to easily compute $H^0(N_{S/\PP^5})= H^0(N_{S/Q^4})\oplus \CC^{20}$ and $H^1(N_{S/\PP^5})=H^2(N_{S/\PP^5})=0$. Hence, ${Hilb}_{\PP^5}^{P(t)}$ is also smooth at $[S]$ and $dim(T_{[S]} {{Hilb}_{\PP^5}^{P(t)}} )=dim (H^0(N_{S/\PP^5}))=58$. 

    Finally, the same arguments above can be applied to prove the statements about surfaces of type $ Z_{B}^{10}$ in $\PP^5$.
\end{proof} 

\begin{pr}\label{open subset parametrizing surfaces of type II} There is a smooth open subset $U\subseteq Hilb_{\PP^5}^{P(t)}$ of dimension $58$ parametrizing surfaces of type $^{II} Z_{E}^{10}$ in $\PP^5$.
\end{pr}
\begin{proof}

    Let $[S]\in {Hilb}_{\PP^5}^{P(t)}$ and consider the short exact sequence 
    \[0 \rightarrow \mathcal{I}_{S/\PP^5}(2)\rightarrow \ko _{\PP^5} (2)\rightarrow \ko _S (2)\rightarrow 0.\] 
    By the additivity property of the Euler characteristic $\chi$ over exact sequences, we have $\chi(\mathcal{I}_{S/\PP^5}(2))+ P(2)=\chi (\ko _{\PP^5} (2))$, or equivalently, $\chi(\mathcal{I}_{S/\PP^5}(2))=1$. In particular, the inequality $h^0(\mathcal{I}_{S/\PP^5}(2))+h^2(\mathcal{I}_{S/\PP^5}(2))\geq 1$ holds. On the other hand, we have $h^2(\mathcal{I}_{S/\PP^5}(2))=0$ if the surface $S$ is of type $ Z_{B}^{10}$ or $ Z_{E}^{10}$ (see the proof of \cite[Proposition~4.3]{gross.degree10}). By the semicontinuity theorem, there is an open subset of $W\subseteq {Hilb}_{\PP^5}^{P(t)}$ containing surfaces of type $ Z_{B}^{10}$ and $ Z_{E}^{10}$ and characterized by the condition $h^2(\mathcal{I}_{S/\PP^5}(2))=0$. In particular, $h^0(\mathcal{I}_{S/\PP^5}(2))\geq 1$ on $W$, and thus we have $W=\bigcup _{V(2)\in |\ko _{\PP^5} (2)| }{Hilb}_{V(2)}^{P(t)}$. Now consider the open subset $V$ of $W$ characterized by the equation $h^0(\mathcal{I}_{S/\PP^5} (2))= 1$. The preimage of the locus of smooth quadrics in $|\ko _{\PP^5} (2)|$ under the canonical morphism $V\rightarrow |\ko _{\PP^5} (2)|$ is open in $V$. Moreover, the locus of smooth surfaces in $Hilb_{\PP^5} ^{P(t)}$ is open as well. Hence, there is an open subset of $Hilb_{\PP^5} ^{P(t)}$ parametrizing smooth surfaces contained in a unique smooth quadric. Thus, these surfaces can only be of type $Z_{B}^{10}$ and $^{II} Z_{E}^{10}$. Since the tangent space dimensions differ from one class of surfaces to the other, we see that there are two disjoint smooth open subsets of ${Hilb}_{\PP^5}^{P(t)}$ parametrizing surfaces of type $Z_{B}^{10}$ and $^{II} Z_{E}^{10}$, respectively. We thus have a smooth open subset $U$ of $Hilb_{\PP^5}^{P(t)}$ of dimension $58$ parametrizing surfaces of type $^{II} Z_{E}^{10}$ by Lemma \ref{Hilb is smooth at S type EII and type B}.
\end{proof}

\section{Special cubic fourfolds of discriminant 14}\label{section 2}

A special cubic fourfold is a smooth cubic hypersurface $X\subseteq \PP ^5$ containing a surface $S$ not homologous to a complete intersection. The saturation of the rank $2$ sublattice of $H^4(X, \ZZ )$ generated by the square of the hyperplane class $h$ of $X$ and the class of the surface $S$ has discriminant $d=3S^ 2 - deg(S)^2$. We also say that $X$ has discriminant $d$. 

The coarse moduli space $\mathcal{C}=\sfrac{\mathbf{sm}|\ko_{\PP^5}(3)|}{PGL(6)}$ of smooth cubic fourfolds in $\PP^5$ is quasi-projective and $20$-dimensional. Special cubic fourfolds in $\mathcal{C}$ are parametrized by a countable family of irreducible divisors $\mathcal{C}_d$, called Hassett divisors, indexed by discriminants $d>6$ such that $d=0\text{ or }2\text{ mod }  6$ by \cite[Theorem~4.3.1]{Hassett.special}.

We aim to show that surfaces of type $^{II} Z_{E}^{10}$ characterize cubics in the Hassett divisor $\mathcal{C}_{14}$. First of all, we need to make sense of our claim by proving the following lemma.

\begin{lem}\label{cubics containing surfaces of type II are special}
    If $S$ is a surface of type $ ^{II} Z_{E}^{10}$ and $X$ is a smooth cubic fourfold that contains $S$, then $X$ is special of discriminant $14$, i.e., $[X]\in \mathcal{C}_{14}$. 
\end{lem}
\begin{proof}
    The surface $S$ is not homologous to a complete intersection, since $S^2=c_2(N_{S / X})=6 deg(S)+3 h K_S+K_S^2-c_2 (S)=(6\times 10)+(3\times 2) -2-26=38$ according to \cite[\S 4]{Hassett.special} and \cite[Proposition~4.3, Table~1]{gross.degree10}, but $38$ is not divisible by $(h^2)^2=3$. Moreover, $\langle h^2, S\rangle$ is a sublattice of $H^4(X,\ZZ)$ of rank $2$ and discriminant $3S^ 2 - deg(S)^2=(3\times 38) -{10}^2=114-100=14$ that is necessarily saturated, since $14$ is a square-free integer. 
\end{proof}

\begin{proof}[(Proof of Theorem \ref{characterization of C14})]
    Consider the Hilbert scheme of flags 
    \[\mathcal{FH}:=\{[S\subseteq X]:[S]\in {Hilb}_{\mathbb{P}^5}^{P
    (t)} \text{ and } [X]\in \mathbf{sm}|\ko_{\PP^5}(3)|\}\subseteq {Hilb}_{\mathbb{P}^5}^{P
    (t)}\times \mathbf{sm}|\ko_{\PP^5}(3)|.\] 
    Note that the first projection $p: \mathcal{FH} \rightarrow {Hilb}_{\mathbb{P}^5}^{P
    (t)}$ has fiber $p^{-1}([S])=\mathbb{P}(H^0(\mathcal{I}_{S / \mathbb{P}^5}(3)))$, while the second projection $q: \mathcal{FH} \rightarrow \mathbf{sm}|\ko _{\PP^5}(3)|$ has fiber $q^{-1}([X])={Hilb}_{X}^{P
    (t)}$.

    Let $[S]\in {Hilb}_{\PP^5}^{P(t)}$. We have the short exact sequence 
     \[0 \rightarrow \mathcal{I}_{S/\PP^5}(3)\rightarrow \ko _{\PP^5} (3)\rightarrow \ko _S (3)\rightarrow 0.\] 
    By the additivity property of the Euler characteristic $\chi$ over exact sequences, we have $\chi(\mathcal{I}_{S/\PP^5}(2))+ P(3)=\chi (\ko _{\PP^5} (3))$, or equivalently, $\chi(\mathcal{I}_{S/\PP^5}(3))=12$. In particular, the inequality $h^0(\mathcal{I}_{S/\PP^5}(3))+h^2(\mathcal{I}_{S/\PP^5}(3))\geq 12$ holds. On the other hand, we have $h^2(\mathcal{I}_{S/\PP^5}(3))=0$ for any surface $S$ of type $ Z_{E}^{10}$ (see the proof of \cite[Proposition~4.3]{gross.degree10}). In particular, we have $h^0(\mathcal{I}_{S/\PP^5}(3))\geq 12$ for all surfaces $S$ of type $^{II} Z_{E}^{10}$. In fact, it is possible to construct an explicit surface \texttt{S} of type $^{II} Z_{E}^{10}$ such that $h^0(\mathcal{I}_{\texttt{S}/\PP^5}(3))=12$ using \texttt{Macaulay2}:
    
\bigskip
{\footnotesize
\begin{Verbatim}[commandchars=&!$]
Macaulay2, version 1.22.0.1
with packages: ConwayPolynomials, Elimination, IntegralClosure, InverseSystems, Isomorphism, 
LLLBases, MinimalPrimes, OnlineLookup, PrimaryDecomposition, ReesAlgebra, Saturation, 
TangentCone, Varieties

&colore!darkorange$!i1 $: &colore!airforceblue$!needsPackage$ "K3Surfaces"

&colore!darkspringgreen$!o1 $= &colore!darkorchid$!K3Surfaces$

&colore!darkspringgreen$!o1 $: Package

&colore!darkorange$!i2 $: S0 = &colore!airforceblue$!K3$(7,5,0)

&colore!darkspringgreen$!o2 $= &colore!darkorchid$!K3 surface with rank 2 lattice defined by the intersection matrix:$                                                       
                                                                         &colore!darkorchid$!| 12 5 |$
                                                                         &colore!darkorchid$!| 5  0 |$
     &colore!darkorchid$!-- (1,0): K3 surface of genus 7 and degree 12 containing elliptic curve of degree 5$ 

&colore!darkspringgreen$!o2 $: Lattice-polarized K3 surface

&colore!darkorange$!i3 $: S = &colore!airforceblue$!project$({1,1},S0(1,0))
-- *** simulation ***
-- surface of degree 12 and sectional genus 7 in PP^7 (quadrics: 10, cubics: 64)
-- surface of degree 11 and sectional genus 7 in PP^6 (quadrics: 5, cubics: 34)
-- surface of degree 10 and sectional genus 7 in PP^5 (quadrics: 1, cubics: 12)

-- (degree and genus are as expected)

&colore!darkspringgreen$!o3 $= &colore!darkorchid$!S$

&colore!darkspringgreen$!o3 $: ProjectiveVariety, surface in PP^5

&colore!darkorange$!i4 $: &colore!airforceblue$!singularLocus$ S

&colore!darkspringgreen$!o4 $= &colore!darkorchid$!empty subscheme of PP^5$

&colore!darkspringgreen$!o4 $: ProjectiveVariety, empty subscheme of PP^5

&colore!darkorange$!i5 $: &colore!airforceblue$!euler$ S

&colore!darkspringgreen$!o5 $= &colore!darkorchid$!26$

&colore!darkorange$!i6 $: Q = &colore!airforceblue$!random$(2,S)

&colore!darkspringgreen$!o6 $= &colore!darkorchid$!Q$

&colore!darkspringgreen$!o6 $: ProjectiveVariety, hypersurface in PP^5

&colore!darkorange$!i7 $: &colore!airforceblue$!describe$ Q

&colore!darkspringgreen$!o7 $= &colore!darkorchid$!ambient:.............. PP^5$
     &colore!darkorchid$!dim:.................. 4$
     &colore!darkorchid$!codim:................ 1$
     &colore!darkorchid$!degree:............... 2$
     &colore!darkorchid$!generators:........... 2^1$
     &colore!darkorchid$!purity:............... true$
     &colore!darkorchid$!dim sing. l.:......... -1$
\end{Verbatim}
} \noindent 

    \bigskip
    Recall that there exists a smooth open set $U\subseteq Hilb_{\PP^5}^{P(t)}$ of dimension $58$ that parametrizes surfaces of type $^{II} Z_{E}^{10}$ in $\PP^5$ by Proposition \ref{open subset parametrizing surfaces of type II}. In view of the previous \texttt{Macaulay2} script, we have a non-empty open subset $U'$ of $U$ such that $h^0(\mathcal{I}_{S/\PP^5}(3))=12$ for all $[S]\in U'$ by the semicontinuity theorem. Hence, $p^{-1}(U')\subseteq \mathcal{FH}$ is a smooth projective bundle over $U'$ and thus
    \[dim_{[S\subseteq X]}(p^{-1}(U'))=dim(T_{[S]}{Hilb}_{\mathbb{P}^5}^{P(t)})+ h^0( \mathcal{I}_{S / \mathbb{P}^5}(3))-1=58+12-1=69\]
    for any flag $[S \subseteq X]\in p^{-1}(U')$.
    
    Note that $q: \mathcal{FH} \rightarrow \mathbf{sm}|\ko_{\PP^5}(3)|$ maps $p^{-1}(U')$ into the preimage of the Hassett divisor $\mathcal{C}_{14}$ under the quotient map $\mathbf{sm}|\ko_{\PP^5}(3)|\rightarrow \mathcal{C}=\sfrac{\mathbf{sm}|\ko_{\PP^5}(3)|}{PGL(6)}$ by Lemma \ref{cubics containing surfaces of type II are special}. Thus, $dim (q(p^{-1}(U'))) \leq dim(\mathcal{C}_{14})+ dim PGL(6)=54$. By applying the theorem on the dimension of the fibers, we have 
\begin{align*}
    h^0( N_{S/X}) &=dim(T_{[S]}Hilb_{X}^{P(t)})\\
    &\geq dim_{[S\subseteq X]} (q^{-1}([X]))\\
    &\geq dim (p^{-1}(U'))-dim (q(p^{-1}(U')))\\
    &\geq dim (p^{-1}(U'))-54 \\
    &=69-54\\
    &=15
\end{align*}
    for any flag $[S \subseteq X]\in q^{-1}([X]) $.

    We can use \texttt{Macaulay2} to find a smooth cubic \texttt{X} containing \texttt{S} such that $h^0( N_{\texttt{S}/\texttt{X}})=15$:

\bigskip
{\footnotesize
\begin{Verbatim}[commandchars=&!$]
&colore!darkorange$!i8 $: X = &colore!airforceblue$!random$(3,S)

&colore!darkspringgreen$!o8 $= &colore!darkorchid$!X$

&colore!darkspringgreen$!o8 $: ProjectiveVariety, hypersurface in PP^5

&colore!darkorange$!i9 $: &colore!airforceblue$!describe$ X

&colore!darkspringgreen$!o9 $= &colore!darkorchid$!ambient:.............. PP^5$
     &colore!darkorchid$!dim:.................. 4$
     &colore!darkorchid$!codim:................ 1$
     &colore!darkorchid$!degree:............... 3$
     &colore!darkorchid$!generators:........... 3^1$
     &colore!darkorchid$!purity:............... true$
     &colore!darkorchid$!dim sing. l.:......... -1$

&colore!darkorange$!i10 $: &colore!airforceblue$!rank$ &colore!airforceblue$!HH$^0(&colore!airforceblue$!sheaf$ &colore!airforceblue$!Hom$((&colore!airforceblue$!ideal$ S/&colore!airforceblue$!ideal$ X)**(&colore!airforceblue$!ring$(&colore!airforceblue$!ideal$ S)/&colore!airforceblue$!ideal$ S),&colore!airforceblue$!ring$(&colore!airforceblue$!ideal$ S)/&colore!airforceblue$!ideal$ S))

&colore!darkspringgreen$!o10 $= &colore!darkorchid$!15$
\end{Verbatim}
} \noindent 

    \bigskip
    This implies that $dim (q(p^{-1}(U')))=54$. Thus, $\overline{q(p^{-1}(U'))}$ fills the preimage of $\mathcal{C}_{14}$ and the result follows.
\end{proof}

\section{Surfaces of type \texorpdfstring{$^{II} Z_{E}^{10}$}{TEXT} and quintic Del Pezzo surfaces}\label{section 3}

If $X$ is a special cubic fourfold of discriminant $14$ containing a quartic scroll $\Sigma$ and a quintic Del Pezzo surface $D$, then it is easy to see that $\Sigma =3h^2 -D$ in $H^4(X,\ZZ)$. In fact, there exists a Segre threefold $\PP^1 \times \PP^2\subseteq \PP^5$, which has degree $3$, that satisfies $(\PP^1 \times \PP^2)\cap X= \Sigma \cup D$ by \cite{Tregub.TwoRmks}. The relation between surfaces of type $^{II} Z_{E}^{10}$ and quintic Del Pezzo surfaces is numerically similar to the case of quartic scrolls and quintic Del Pezzo surfaces. However, we can show that these surfaces cannot be residual to each other via some smooth threefold of $\PP^5$.

\begin{proof}[(Proof of Theorem \ref{surfaces of type II are not linked to quintic DP})]
    There exists an isometry $\gamma$ of $H^4(X, \ZZ)$ preserving $h^2$ such that $\langle h^2,D \rangle = \gamma (\langle h^2,S \rangle)$ by \cite[Proposition~3.2.4]{Hassett.special}. In particular, $\gamma$ defines an integral invertible matrix of the form $\begin{pmatrix} 1 &a \\ 0 &b \end{pmatrix}$ with respect to the bases $(h^2,S)$ and $(h^2,D)$. In other words, we have $S=ah^2+bD$ in $H^4(X,\ZZ)$, where $b=\pm 1$ and $a\in \ZZ$. Then $10=Sh^2=3a+5b$, and thus we see that $b=-1$ and $a=5$. Hence, $S=5h^2-D$ in $H^4(X,\ZZ)$.
    
    Note that a threefold $Y$ as in the statement must be non-degenerate, since $S$ is. By \cite[Proposition~1.5]{Okonek.3foldsinP5}, a non-degenerate smooth quintic threefold $Y\subseteq \PP^5$ is a so-called Castelnuovo threefold. This is a quadric bundle over $\PP^1$ with $15$ singular fibers. Moreover, its ideal sheaf fits in the short exact sequence
    \[0\rightarrow \ko_{\PP^5}(-4)^{\oplus 2}\rightarrow \ko_{\PP^5}(-3)^{\oplus 2}\oplus \ko_{\PP^5}(-2)\rightarrow \mathcal{I}_{Y/\PP^5}\rightarrow 0.\]
    From this resolution it is clear that $Y$ is $(2,3)$-linked to a $\PP^3$, e.g., there exist quadric and cubic hypersurfaces $Z$ and $X'$ in $\PP^5$, respectively, such that $Z\cap X' =Y\cup \PP^3$. In particular, both $Z$ and $X'$ are necessarily singular. Since a surface $S$ of type $^{II} Z_{E}^{10}$ is contained in a single quadric hypersurface of $\PP^5$, and this quadric is smooth, we see that $S$ cannot be contained in $Y$. In particular, $S$ and $D$ are never linked in $X$ via a smooth quintic $Y$.
\end{proof}

\end{document}